\documentclass{amsart}
\usepackage{amssymb}
\usepackage{amsmath}
\usepackage{subfigure}
\usepackage{epsfig}
\usepackage[all]{xy}

\newtheorem{theorem}{Theorem}[section] % 1st argument is your name for it
\newtheorem{lemma}[theorem]{Lemma}     % 2nd argument is what is printed

\newtheorem{proposition}[theorem]{Proposition}

\theoremstyle{definition}
\newtheorem{example}[theorem]{Example}
\newtheorem{remark}[theorem]{Remark}

\title[Simultaneous diagonalization of vector bundles]
{Simultaneous diagonalization of vector bundles}

\author{Luis Arenas-Carmona}

\newcommand\ad{\mathbb{A}}

\newcommand\alge{\mathfrak{A}}

\newcommand\oink{\mathcal O}

\newcommand\enteri{\mathbb Z}

\newcommand\Da{\mathfrak{D}}

\newcommand\Ea{\mathfrak{E}}
\newcommand\Fa{\mathfrak{F}}

\newcommand\bmattrix[4]{\left(\begin{array}{cc}#1&#2\\#3&#4\end{array}\right)}
\newcommand\sbmattrix[4]{\textnormal{\scriptsize$\left(\begin{array}{cc}#1&#2\\#3&#4\end{array}\right)$\normalsize}}

\newcommand\matrici{\mathbb{M}}
\newcommand\finitum{\mathbb{F}}

\begin{document}
\maketitle

\begin{abstract}
Grothendieck-Birkhoff Theorem states that every finite dimensional vector bundle over the projective line splits as the 
sum of one dimensional vector bundles. In this work we study simultaneous splittings of two dimensional vector
bundles over a finite field using the theory of Eichler orders. In the sheaf-theoretical context, an Eichler
order in a matrix algebra is the intersection of the sheaves of endomorphisms  of two vector bundles,
so characterizing split Eichler orders solves the problem of simultaneous splitting. 
 We caracterize both the genera of Eichler orders containing only split Eichler orders
and the genera containing only a finite number of non-split clases, in terms of a divisor-valued distance 
parametrizing the genera.
\end{abstract}

\section{Introduction}\label{intro}

In all of this work $K$ is a global function field, i.e., the field of rational functions on a smooth irreducible
projective curve $X$ over a finite field. We let $\oink_X$ denote the structure sheaf of $X$. 
Let $\alge=\mathbb{M}_2(K)$ be the $4$-dimensional matrix algebra, while $\Da$ and $\Ea$ denote 
$X$-orders of maximal rank in $\alge$, as defined below. 

An $X$-lattice 
$\Lambda$ is a locally free sheaf of  $\oink_X$-modules of finite rank $n$.
The sheaf of sections of a vector bundle is an $X$-lattice, and as usual we identify the bundle and the lattice. 
The generic fiber $\Lambda\otimes_{\oink_X}K$ is isomorphic to $K^n$ as a vector space over $K$,
 and we fix one such isomorphism by saying that $\Lambda$ is a lattice in $K^n$. Equivalently, we choose
a $K$-linearly independent set of $n$  sections over some afine subset $U_0\subset X$ and identify it with the canonical basis
of $K^n$. 
This implies that the group of $U$-sections $\Lambda(U)$ is identified with a subset of $K^n$ for any
open set $U\subseteq X$. 
  Thus defined, two lattices $\Lambda$ and $\Lambda'$, or their corresponding bundles,
are isomorphic if and only if there exists an invertible n-by-n matrix $T$ satisfying $T\Lambda=\Lambda'$.  
 Similar conventions applies to other explicit vector spaces.
 Note that $\Lambda(U)$ is a lattice over the Dedekind domain $\oink_X(U)$
as defined in \cite{Om}.
An order $\Da$ in a $K$-algebra $\alge$ is an $X$-lattice in $\alge$ such that $\Da(U)$ is a ring for any
open subset $U$, e.g., the structure sheaf $\oink_X$ is an $X$-order in $K$. %Equivalently,
% an $X$-order in an algebra $\alge$ is a sheaf of $\oink_X$-algebras whose generic fiber is $\alge$.  
We let  $\finitum=\oink_X(X)$ denote the full constant field  of $K$.
 The group of global sections $\Lambda(X)$ is a finite dimensional vector space over $\finitum$ for any 
$X$-lattice $\Lambda$. 

 Recall that every $X$-bundle in the one dimensional space $K$  has the form 
$$\mathfrak{L}^B(U)=\left\{f\in K\Big|\mathrm{div}(f)|_U+B|_U\geq0\right\},$$
for some fixed divisor $B$ on $X$, and for every open set $U\subseteq X$.
These bundles are usually called invertible bundles in current literature, and have the following properties:
\begin{enumerate}
\item  Linearly equivalent divisors define isomorphic bundles,
\item $\mathfrak{L}^B\mathfrak{L}^D= \mathfrak{L}^{B+D}$, for every pair of divisors $(B,D)$,
\item $\mathfrak{L}^B(U)\subseteq\mathfrak{L}^D(U)$ for every open set $U$ if and only if
$B\leq D$.
\end{enumerate}
In (b), $\mathfrak{L}^B\mathfrak{L}^D$ denotes the sheaf defined by $(\mathfrak{L}^B\mathfrak{L}^D)(U)=
\mathfrak{L}^B(U)\mathfrak{L}^D(U)$ on open sets $U\subseteq X$, which is isomorphic to the
tensor product $\mathfrak{L}^B\otimes_{\oink_X}\mathfrak{L}^D$. In higher dimensions, similar conventions apply
 to scalar products or other bilinear maps. 

A split $X$-lattice or split $X$-bundle is a lattice isomorphic to a direct sum of invertible bundles, 
e.g., a two dimensional $X$-lattice $\Lambda$ is split if $\Lambda\cong\mathfrak{L}_1\times\mathfrak{L}_2$,
as $\oink_X$-modules, where $\mathfrak{L}_1$ and $\mathfrak{L}_2$ are  invertible bundles.
We say that a basis $\{e_1,e_2\}$ splits or diagonalizes an $X$-bundle $\Lambda$ in $K^2$ if 
$\Lambda=\mathfrak{L}_1e_1\oplus\mathfrak{L}_2e_2$ where $\mathfrak{L}_1$ and $\mathfrak{L}_2$
are invertible bundles. Certainly a bundle in $K^2$ is split if and only if  it is split by at least one basis.

Since an isomorphism between $\Lambda\otimes_{\oink_X}K$ and $K^2$, as constant bundles, is given by 
the choice of two linearly independent sections $v_1$ and $v_2$,
 finding a split basis for $\Lambda$ is equivalent to finding
a matrix $T=\sbmattrix abcd\in\matrici_2(K)$ satisfying
$$\Lambda=\mathfrak{L}_1(av_1+cv_2)+\mathfrak{L}_2(cv_1+dv_2),$$
where $\mathfrak{L}_1$ and $\mathfrak{L}_2$ are  invertible bundles. The problem of simultaneous splitting
consist on finding a common  matrix that is useful for both bundles. This is an analog for bundles over projective
curves of the theory
of invariant factors for dedekind domains. Our results shows that simultaneous diagonalization is
not always possible, but it is always possible on the projective line, for two bundles that are "very close", 
in a sense we define shortly.

 To each $X$-bundle $\Lambda$ in $K^2$ we associate
the order $\Da_{\Lambda}=\mathcal{E}\mathit{nd}_{\oink_X}(\Lambda)$
 in the matrix algebra $\matrici_2(K)$, defined by 
$$\Da_\Lambda(U)=\left\{a\in \matrici_2(K)\Big|a\Lambda(U)\subseteq\Lambda(U)\right\},$$
for every open set $U\subseteq X$. This is a maximal order in the matrix algebra and every maximal order
of this algebra has the form $\Da_\Lambda$ for some $X$-bundle $\Lambda$ in $K^2$.  
 The $X$-bundle $\Lambda$ is split by a certain basis $\{e_1,e_2\}$
if and only if the corresponding maximal order has the form 
$\Da_\Lambda=\sbmattrix {\oink_X}{\mathfrak{L}^{-D}}{\mathfrak{L}^{D}}{\oink_X}$, for
some divisor $D$ in that basis. In fact, if $\Lambda=\mathfrak{L}^Be_1\oplus\mathfrak{L}^Ce_2$,
we can choose $D=B-C$. This condition on $\Da_\Lambda$
 is equivalent to $\sbmattrix 1000, \sbmattrix 0001\in\Da_\Lambda(X)$.
More generally, we say that an order $\Ea$ is split if 
$\Ea$ is conjugate to $\sbmattrix {\oink_X}{\mathfrak{L}_1}{\mathfrak{L}_2}{\oink_X}$ for some pair of
invertible bundles $(\mathfrak{L}_1,\mathfrak{L}_2)$, or equivalently,
 if its ring of global sections contain a non-trivial idempotent. A split order is split as a lattice but
the converse is false in general.

 It is well known that two vector bundles $\Lambda$ and
$M$ satisfy $\Da_\Lambda=\Da_M$ if and only if there exists an invertible  vector bundle 
$\mathfrak{L}$ such that $\Lambda=\mathfrak{L}M$, where the product on the right is the scalar product.
It is apparent that every basis splitting $M$ splits also $\mathfrak{L}M$, so the problem of simultaneous splitting
can be studied in terms of  pairs of maximal orders rather than pairs of vector bundles.
For every pair of maximal orders  $\Da_\Lambda$ and $\Da_{\Lambda'}$ we define the Eichler order
$\Ea_{\Lambda,\Lambda'}=\Da_\Lambda\cap\Da_{\Lambda'}$. This is an order of maximal rank in $\matrici_2(K)$.
By the previous discussion, there exists a basis splitting both maximal orders simultaneously if and only
if  $\Ea_{\Lambda,\Lambda'}(X)$ contains a global idempotent, i.e., it is split.
 Characterizing pairs of bundles that can be diagonalized simultaneously is, therefore, equivalent to characterizing
 split Eichler orders.

As we recall in \S2, an order of maximal rank in $\matrici_2(K)$, or more generally a lattice $\Lambda$
 in a vector space $V$, is completely determined  by its set of completions 
$\left\{\Lambda_P\subseteq V_P\Big| P\in|X|\right\}$, where $|X|$ denote the set of closed points of $X$.
Such orders are usually classfied locally through the concept of genus, a maximal set of locally isomorphic  lattices.
In our context, Eichler orders are classified into genera by its level. The level of an Eichler order
 $\Ea_{\Lambda,\Lambda'}$ is an efective divisor $D=D(\Da_\Lambda,\Da_{\Lambda'})$
such that, for every affine open set $U\subseteq X$, we have an isomorphism of $\oink_X(U)$ modules
$$\Da_\Lambda(U)/\Ea_{\Lambda,\Lambda'}(U)\cong
\Da_{\Lambda'}(U)/\Ea_{\Lambda,\Lambda'}(U) \cong\oink_X(U)/\mathfrak{L}^{-D}(U).$$
In terms of this distance, our main results are as follows:

\begin{theorem}\label{t1}
For an arbitrary smooth projective curve $X$ over a finite field, and for any effective divisor $D$, 
there is only a finite number of conjugacy classes of non-split Eichler orders of level $D$
 if and only if $D$ is multiplicity free, i.e., $D=\sum_{i=1}^nP_i$, where $P_1,\dots,P_n$
are different points.
\end{theorem}

When $X=\mathbb{P}^1$ is the projective line, Grothendick-Birkhoff  Theorem \cite{burban}
 states that every vector bundle over $X$ splits as a sum of line bundles.
 In our pressent setting, this implies that every maximal
order splits, whence next result can be considered a partial generalization.

\begin{theorem}\label{t2}
Assume $X=\mathbb{P}^1$ is the projective line. Then the following statements are equivalent for any 
effective divisor $D$:
\begin{enumerate}
\item Every Eichler orders of level $D$ is split,
 \item $D$ is the sum of at most two different points of degree 1. 
\end{enumerate}
\end{theorem}

The main tool in the sequel is the concept of quotient graph, specifically
quotients of the local Bruhat-Tits tree at some place $P$.
 This idea is due to J.-P. Serre who studied the relation between these
quotients and the structure of the arithmetic groups defining them \cite{trees}. These groups are usually
unit groups of orders, and the corresponding quotient is the S-graph, as defined at the end of \S3.
 In fact, Serre himself computed the S-graph when $X=\mathbb{P}^1$ and $P$ is a place of degree
 $4$ or less, using tools from algebraic geometry.
 An elementary proof of Serre's result was given in \cite{Mason1}, and some partial generalizations
appear in \cite{Mason2} and \cite{Mason6}. These quotients have been used
to study non-congruence subgroups of Drinfeld modular groups, see \cite{Mason3} or \cite{Mason4}.
We ourselves in \cite{cqqgvro} gave a recursive formula to compute these graphs for a place $P$ of arbitrary
degree using the theory of spinor genera, and introduced the concept of C-graph that is
used here for the study of conjugacy classes in a genus.  A close formula for all S-graphs for
maximal orders on $\mathbb{P}^1$ has been given by R. K\H ohl, B. M\H uhlherr and K. Struyve 
in \cite{Kohl} by a different method involving double cosets for simultaneous 
actions on two trees. The S-graph has also been computed for places of degree 1 on
 an elliptic curve \cite{takahashi}.  M. Papikian has studied the S-graph when $\matrici_2(K)$ is replaced
by a division algebra \cite{Papikian}. Although the theory only requires the orders to
be maximal at the specific place $P$, as far as we can tell the present work is the first attempt 
to use these graphs to study Eichler orders, or any non-maximal order, on a function field.

\section{Completions and spinor genera}\label{ndos}

In this section we review the basic facts about spinor genera and spinor class fields of orders.
See \cite{abelianos} for details.

Let $|X|$ be the set of closed points in the smooth projective curve $X$.
 For every such point $P\in |X|$ we let $K_P$ be the completion at $P$ of the function field $K=K(X)$. 
We denote by $\ad=\ad_X$ the adele ring of $X$, i.e., the subring of elements $a=(a_P)_P\in\prod_{P\in|X|}K_P$,
for which all but a finite number of coordinates $a_P$ are integral. For any finite dimensional 
vector space $V$ over $K$ we define its adelization $V_\ad=V\otimes_K\ad\cong\ad^{\mathrm{dim}_KV}$
 and give it the adelic topology \cite[\S IV.1]{weil}.  Note that $K_\ad\cong\ad$ canonically.
We identify the ring of $\ad$-linear maps $\mathrm{End}_\ad(V_\ad)$ with the adelization $\big(\mathrm{End}_K(V)\big)_\ad$.
For any $X$-lattice $\Lambda$, the completion $\Lambda_P$ is defined as the closure of $\Lambda(U)$
in $K_P$ for an arbitrary affine open set $U$ containing $P$. This definition is independent of the choice of $U$.
 Note that for every affine subset $U\subseteq X$,
the $\oink_X(U)$-module $\Lambda(U)$ is an $\oink_X(U)$-lattice as defined in \cite{Om}.
 In this work a lattice always means an $X$-lattice or $X$-bundle as in \S1, while we use 
affine lattice for the classical concept.
 The same observation apply to orders and the notations $\Da$ and $\Da(U)$.
Just as in the affine case, $X$-lattices are determined
 by their local completions $\Lambda_P$, where $P$ runs over the set $|X|$,
 in the following sense:
\begin{enumerate}
\item For any two lattices $\Lambda$ and $\Lambda'$ in a vector space $V$, $\Lambda_P=\Lambda'_P$
 for almost all $P$,
\item if $\Lambda_P=\Lambda'_P$ for all $P$, then $\Lambda=\Lambda'$, and 
\item every family $\{\Lambda''(P)\}_P$ of local lattices satisfying  $\Lambda''(P)=\Lambda_P$ for almost all $P$, 
is the family of completions of a global lattice $\Lambda''$ in $V$.
\end{enumerate}
The same results apply to orders. 
We also define the adelization $\Lambda_\ad=\prod_{P\in|X|}\Lambda_P$, which is open and compact as
a subgroup of $V_\ad$. This applies in particular to the ring of integral ideles $\oink_\ad=(\oink_X)_\ad
\subseteq K_\ad=\ad$.
It follows from property (3) above that every open and compact $\oink_\ad$-sub-module of $V_\ad$ is 
the adelization of a lattice.
 For every $X$-lattice $\Lambda$ and every element $a\in \mathrm{End}_\ad(V_\ad)$,
 the adelic image $L=a\Lambda$ is the unique  $X$-lattice satisfying $L_\ad=a\Lambda_\ad$. 
The adelic image $L$ thus defined inherit all local properties of the original $X$-lattice $\Lambda$. 
For instance,  adelic images of orders and maximal orders, under conjugation, are orders
and maximal orders, respectively. In particular, if we fix a maximal $X$-order $\Da$, all maximal $X$-orders
in $\alge$ have the form $\Da'=a\Da a^{-1}$ for $a\in\alge_\ad^*$. This conjugation must be interpreted as an adelic image. More generally, a maximal set of adelically conjugate orders is called a genus. 
The set of maximal $X$-orders is a genus \cite{Eichler2}.

Locally, there is a well defined distance $d_P$ between maximal orders in $\matrici_2(K_P)$. in fact, 
we have $d_P(\Da_P,\Da'_P)=d$ if, in some basis,  both orders take the form
$$\Da_P=\bmattrix {\oink_P}{\oink_P}{\oink_P}{\oink_P},\qquad
\Da'_P=\bmattrix {\oink_P}{\pi_P^d\oink_P}{\pi_P^{-d}\oink_P}{\oink_P},$$
where $\pi_P$ is a local uniformizing parameter. 
Intersections of orders can be computed locally, in the sense that $\Da_P\cap\Da'_P=(\Da\cap\Da')_P$
for every pair of orders. We define an Eichler order as
the intersection of two maximal orders. This is certainly a local property. 
The level of a local Eichler order is by definition the distance between the
maximal orders defining it. In the local setting, there is a unique pair of maximal orders whose intersection is a
given Eichler order.  Two local Eichler orders are conjugate if and only if their levels coincide.
We conclude that two global Eichler orders $\Ea$ and $\Ea'$ belong to the same genus
if the local levels coincide at all places. Globally, the distance between two maximal orders
$\Da$ and $\Da'$ is defined as the effective divisor
$$D=D(\Da,\Da')=\sum_{P\in|X|}d_P(\Da_P,\Da'_P)P.$$
If $\Da=\Da_\Lambda$ and $\Da'=\Da_{\Lambda'}$, then  $D=D(\Da_\Lambda,\Da_{\Lambda'})$
is, by definition, the level  $\lambda(\Ea_{\Lambda,\Lambda'})$
 of the Eichler order $\Ea_{\Lambda,\Lambda'}$. Two Eichler order
belong to the same genus if and only if they have the same level.
The genus of Eichler orders of level $D$, for any effective divisor $D$, is denoted $\mathbb{O}_D$.

   Two $X$-orders of maximal rank  $\Da$ and $\Da'$ in $\matrici_2(K)$ are in the same spinor genus if 
$\Da'=a\Da a^{-1}$
for some element $a=bc$ where $b\in\matrici_2(K)$ and $c$ is an adelic matrix satisfying $\mathrm{det}(c)=1_\ad$. 
We write $\Da'\in\mathrm{Spin}(\Da)$ in this case. Equivalently,
 two orders $\Da$ and $\Da'$ are in the same spinor genus
if and only if the rings $\Da(U)$ and $\Da'(U)$ are conjugate for every affine open subset $U\subseteq X$.
 The set of spinor genera in a genus is classified via
the spinor class field, which is defined as the class field corresponding to the group 
$K^*H(\Da)\subseteq \ad^*=:J_X$, where
$$H(\Da)=\{\mathrm{det}(a)|a\in\matrici_2(\ad)^*,\  a\Da a^{-1}=\Da\}.$$
This field depends only on the genus $\mathbb{O}=\mathrm{gen}(\Da)$ of $\Da$. We denote it $\Sigma(\mathbb{O})$.

Let $t\mapsto [t,\Sigma/K]$ denote the Artin map on ideles. There exists a distance map
$$\rho:\mathbb{O}\times\mathbb{O}\rightarrow\mathrm{Gal}\big(\Sigma(\mathbb{O})/K\big),$$
defined by $\rho(\Da,\Da')=[\det(a),\Sigma/K]$, for any adelic element $a\in\alge_\ad^*$
satisfying $\Da'=a\Da a^{-1}$. This distance map has the following properties:
\begin{enumerate}
\item $\rho(\Da,\Da'')=\rho(\Da,\Da')\rho(\Da',\Da'')$ for any triple $(\Da,\Da',\Da'')\in\mathbb{O}^3$,
\item $\rho(\Da,\Da')=\mathrm{Id}_{\Sigma(\mathbb{O})}$ if and only if  $\Da'\in\mathrm{Spin}(\Da)$.
\end{enumerate}
In particular, for the genus $\mathbb{O}_0$ of maximal orders, the corresponding
distance $\rho_0:\mathbb{O}_0^2\rightarrow\mathrm{Gal}\big(\Sigma_0/K\big)$,
where $\Sigma_0=\Sigma(\mathbb{O}_0)$, is related to the divisor valued distance by the formula
$\rho_0(\Da_{\Lambda},\Da_{\Lambda'})=[[D(\Lambda,\Lambda'),\Sigma_0/K]]$. 
where $D\mapsto [[D,\Sigma_0/K]]$ is the Artin map on divisors.
 Note however that the distance $\rho_0$ is trivial between isomorphic orders,
which does not hold for  the divisor valued distance.

The spinor class field $\Sigma(D)=\Sigma(\mathbb{O}_D)$, for Eichler orders of level $D=\sum_Pa_PP$,
 is the maximal subfield of $\Sigma_0$, splitting at every place $P$ for which
$a_P$ is odd. The corresponding distance $\rho_D$ is given by the formula
$$\rho_D(\Ea_{\Lambda,\Lambda'},\Ea_{L,L'})=\rho_0(\Da_\Lambda,\Da_L)\Big|_{\Sigma(D)}.$$
The preceding formula follows from \cite[Prop. 6.1]{scffgeo} and the discussion thereafter.

\begin{remark}
When $\matrici_2(K)$ is replaced by another quaternion algebra $\alge$, the condition for two orders to be
in the same spinor genus goes as follows:
\emph{The orders $\Da$ and $\Da'$ are in the same spinor genus if and only if
$\Da(U)$ and $\Da'(U)$ are conjugate for any open set $U$ whose complement has at least one place
splitting $\alge$} (c.f. \cite[\S2]{abelianos}). For a matrix algebra, this is equivalent to $U\neq X$.
\end{remark}

\section{Eichler orders and trees}\label{ebundles}

In all of this work, a graph $G$ is a pair of sets $V=V(G)$ and $E=E(G)$,
 called the vertex set and the edge set, together
with three functions $s,t:E\rightarrow V$ and $r:E\rightarrow E$ called the source, target, and reverse,
satisfying $$r(a)\neq a,\quad r\big(r(a)\big)=a,\textnormal{ and }s\big(r(a)\big)=t(a)$$ for every edge $a$. 
A simplicial map $\gamma:G\rightarrow G'$ between graphs is a pair of functions
$\gamma_V:V(G)\rightarrow V(G')$ and $\gamma_E:E(G)\rightarrow E(G')$
 preserving these maps, and a similar convention applies
to group actions. A group $\Gamma$ acts
on a graph $G$ without inversions if
 $g.a\neq r(a)$ for every edge $a$ and every element $g\in\Gamma$. An action without inversions defines
a quotient graph. Basse-Serre Theory allows us to determine the structure of the group $\Gamma$ in terms of the quotient graphs in some cases, see \cite{trees} for an account of this subject. If the action has inversions we can still define a quotient graph by replacing $G$ by its barycentric subdivision and ignoring the new vertices unless they become 
endpoints in the quotient, where they are called virtual endpoints, see \cite{cqqgvro} or \cite{rouidqo} for details.
The edge joining a vertex and a virtual endpoint is called a half-edge. It can be interpreted as a vertex that
has been "folded in half" by the action. 

The real-line graph $\mathfrak{R}$ is defined by a collection of vertices $\{n_j|j\in\enteri\}$ and a collection of
edges $\{a_j,r(a_j)|j\in\enteri\}$ satisfying $s(a_j)=n_j$ and $t(a_j)=n_{j+1}$.  An integral interval is
a connected subgraph of $\mathfrak{R}$. A finite integral interval $\mathfrak{I}_{k,k'}$ is completely determined by
its first vertex $n_k$ and it last vertex $n_{k'}$. Its length is $k'-k$.
The notations $\mathfrak{I}_{-\infty,k}$,  $\mathfrak{I}_{k,\infty}$, and  
$\mathfrak{I}_{-\infty,\infty}$,  are defined analogously.
 A path in a graph $G$ is an injective simplicial map $\alpha:\mathfrak{I}\rightarrow G$
where $\mathfrak{I}$ is an integral interval. A path is finite of length $k$, or infinite in one or two directions, if
so is the corresponding integral interval. In general, we identify a simplicial map
 $\gamma:J_{k,k'}\rightarrow G$ with any shift, i.e., any map
 $\gamma_t:J_{k+t,k'+t}\rightarrow G$ satisfying $(\gamma_t)_E(a_{r+t})=\gamma_E(a_r)$.
The reverse of a simplicial map $\gamma:J_{0,2}\rightarrow G$  is the map $\gamma':J_{0,2}\rightarrow G$
satisfying $\gamma'_E(a_1)=\gamma_E\Big(r(a_0)\Big)$ and $\gamma'_E(a_0)=\gamma_E\Big(r(a_1)\Big)$.

Locally, maximal orders in $\matrici_2(K_P)$, or equivalently 
homothety classes of lattices in $K_P^2$, are in correspondence with the vertices of the Bruhat-Tits tree
 $\mathfrak{T}(K_P)$ for $\mathrm{PSL}_2(K_P)$ \cite{trees}. The vertices corresponding to two maximal orders
are neighbors if and only if their local distance, as defined in \S2, is $1$. In this setting, local Eichler orders 
$ \Ea$ of level $k$ are in correspondence with finite paths
$\gamma:\mathfrak{I}_{0,k}\rightarrow\mathfrak{T}(K_P)$. In fact, for every pair of vertices in the tree
there is a unique path conecting them. If we denote by $\Da_v$ the maximal order corresponding
to the vertex $v$, the Eichler order corresponding to a path $\gamma$ as above is $\Ea_\gamma=
\Da_{\gamma_V(n_0)}\cap\Da_{\gamma_V(n_k)}$. The
orders $\Da_{\gamma_V(n_i)}$ for $i\in\{1,\dots,k\}$ are exactly the maximal orders containing $\Ea_\gamma$,
or in the notations of \cite{scffgeo}, we write $\mathfrak{S}_0(\Ea_{\gamma})=\gamma(\mathfrak{I}_{0,k})$.
In particular, the pair of local orders whose intersection is a given Eichler order is unique. 

For a global Eichler order $\Ea$ of level $D=\lambda(\Ea)=\sum_P\alpha_PP$, 
the set of maximal orders containing $\Ea$ is in correspondence with the set of vertices in the finite grid
$\mathbb{S}(\Ea)=\prod_P\mathfrak{S}_0(\Ea_P)$, where $P$ runs over the set of places at which 
$\alpha_P>0$.
Any vertex $v$ of this grid corresponds to a global maximal order $\Da_v$ containing $\Ea$ and conversely.
For any pair $(v_1,v_2)$ of opposite vertices of this grid, the corresponding maximal orders satisfy
$\Ea=\Da_{v_1}\cap\Da_{v_2}$, and for all these pairs the divisor valued distance defined in \S2 is $D$. 
These grids are seen as subsets of a suitable product of Bruhat-Tits trees.
 Fix an effective divisor $D=\sum_P\alpha_PP$. 
Any grid of the form $\mathbb{S}(\Ea)$ for $\lambda(\Ea)=D$  is called an actual $D$-grid. 
Note that $\mathrm{PGL}_2(K)$ acts by conjugation on the set of actual $D$-grids. Conjugacy classes
of actual $D$-grids are called admissible $D$-grids. Next result is immediate from the definitions:

\begin{proposition}
For any efective divisor $D$,
the set of conjugacy classes of Eichler orders of level $D$ in $\mathbb{M}_2(K)$
 is in correspondence with the set of admisible $D$-grids. 
\end{proposition}

If we write $D=D'+\alpha_PP$, where $D'$ is supported away from $P$, the actual $D$-grid $\mathbb{S}(\Ea)$
is a paralellotope having two actual $D'$-grids as opposite faces. These are called the $P$-faces of the $D$-grid.
The $P$-faces of an admisible $D$-grid are well defined as admisible grids. This convention is used in all that follows.

Now let $Q\in |X|$ and let $\mathbb{O}$ be a genus of orders of maximal rank that are maximal at $Q$. Let $U$
be the complement of $\{Q\}$ in $X$. Fix an order $\Da\in \mathbb{O}$, and let $\Psi$ be the set of orders
$\Da'\in\mathbb{O}$ satisfying $\Da'(U)=\Da(U)$. An order $\Da'\in\Psi$ is completely determined by the 
local order $\Da'_Q$, and the set of conjugacy classes of these orders is in correspondence with the vertices
of the classifying graph $C_P(\Da)=\Gamma\backslash\mathfrak{T}(K_Q)$, where $\mathfrak{T}(K_Q)$ is the local
Bruhat-Tits tree at $Q$, and $\Gamma$ is the stabilizer of $\Da(U)$ in $\mathrm{PGL}_2(K)$. 
 As orders in the same spinor genera restrict to conjugate orders in every affine subset,
every conjugacy class in a given spinor genus $\mathrm{Spin}(\Da)$ corresponds 
to a unique vertex in $C_P(\Da)$. The orders in $\Psi$ belong two one or two spinor 
genera, according two whether  $[[Q,\Sigma(\mathbb{O})/K]]$ is trivial or not, and in 
the later case the quotient graph is bipartite. The classifying graph $C_P(\mathbb{O})$ is defined
as the disjoint union of the graphs corresponding to all spinor genera or pairs of spinor genera. 
Note that this is a straightforward generalization of the definition in \cite{cqqgvro}.

Neighbors in the Bruhat-Tits tree, interpreted as orders in $\Psi$ have equal completions at all places
outside $\{Q\}$, so by interpreting maximal orders containing an Eichler order as vertices in a grid as before, next
result follows:

\begin{proposition} Let $D$ be an efective divisor supported away from the place $Q$.
The vertices of the classifying graph $C_Q(\mathbb{O}_D)$
 are in correspondence with the admissible $D$-grids, while its pairs of mutually reverse
edges are in correspondence with the set of admissible $(D+Q)$-grids. The borders of an edge are the vertices
corresponding to the $Q$-faces of the grid corresponding to that edge.  
\end{proposition}

For any divisor $D=\sum_P\alpha_PP$, its absolute value is defined by $|D|=\sum_P|\alpha_P|P$.

\begin{lemma}
Let $\Ea$ be a split Eichler order of level $D$
that can be written as the intersection of two maximal orders isomorphic
to $\Da_B$ and $\Da_{B'}$, then there are divisors $B_0$ and $B'_0$ such that:
\begin{enumerate}
\item $B_0$ is linearly equivalent to $B$ or $-B$,
\item $B'_0$ is linearly equivalent to $B'$ or $-B'$, and
\item $|B_0-B'_0|=D$.
\end{enumerate} 
\end{lemma}

\begin{proof}
Recall that the set of local maximal orders containing a given idempotent, say $\sbmattrix 1000$ is a maximal
path in the corresponding local tree \cite[Cor. 4.3]{Eichler2}. 
Globally, the set of such orders coinciding outside some finite set $S$ of places,
is in correspondence with an infinite grid whose dimension is the cardinality of $S$. Algebraically, they can be described
as the orders of the form $\Da_B=\sbmattrix {\oink_X}{\mathfrak{L}^B}{\mathfrak{L}^{-B}}{\oink_X}$ where 
$B$ is a divisor supported in $S$. It follows that the Eichler orders containing $\sbmattrix 1000$ 
as a global section are the orders 
of the form $\Ea[B,B']=\sbmattrix {\oink_X}{\mathfrak{L}^{-B'}}{\mathfrak{L}^{-B}}{\oink_X}$ where 
$B+B'$ is an effective divisor. In fact, if we define $$G=\sum_P\mathrm{min}\{\alpha_P,\alpha'_P\}P, \qquad
M=\sum_P\mathrm{max}\{\alpha_P,\alpha'_P\}P,$$
where $B=\sum_P\alpha_PP$ and $B'=\sum_P\alpha'_PP$ ,
we have $\Da_B\cap\Da_{B'}=\Ea[M,-G]$, which is an Eichler order of level $M-G$. We
 note that all pairs $(v'',v''')$ of opposite corners, of the grid corresponding to $\Ea[M,-G]$, 
satisfy $\Da_{v''}=\Da_{B''}$ and $\Da_{v'''}=\Da_{B'''}$, where $|B''-B'''|=M-G$. Now the result follows from
\cite[Prop 4.1]{cqqgvro} and the discussion preceding it.
\end{proof}

\begin{lemma}\label{l34}
Let $Q\in |\mathbb{P}^1|$ be a point of degree 2 or larger. Then there exists non-split orders in 
$\mathbb{O}_{Q}$.
\end{lemma}

\begin{proof}
The conjugacy classes of maximal orders in $\mathbb{M}_2(K)$ are the classes $[\Da_{nP}]$ for $n=0,1,2,\dots$
and $P$ a place of degree 1 \cite[\S II.2.3]{trees}. Recall that $Q$ is, as a divisor, linearly equivalent to $dP$ where 
$d=\mathop{\mathrm{deg}}(Q)$.
We need to recall some properties of the classifying graph $C_Q(\mathbb{O}_0)$ of maximal orders:
\begin{enumerate}
\item For any $n>0$, the orders $\Da_{nP},\Da_{(n+d)P},\Da_{(n+2d)P},\dots$ are located in an infinite ray
(c.f. \cite[Th. 1.2]{cqqgvro}).
\item The graph has one connected component if $d$ is odd and two if $d$ is even (c.f. \cite[Th. 1.3]{cqqgvro}).
\end{enumerate} 
When $d>2$, there must exists an edge connecting two orders isomorphic to $\Da_{nP}$ and $\Da_{mP}$,
where neither $n+m$ nor $n-m$ is divisible by $d$.  In particular, if the corresponding Eichler order were split, 
there should exist  two divisors $B_0$ and $B'_0$ of degrees $\pm n$ and $\pm m$
satisfying $|B_0-B'_0|=Q$. This can only mean $B_0-B'_0=\pm Q$, which is not possible by degree considerations,
 and the result follows from the preceding lemma. Assume now $d=2$.
We learn from \cite[Fig. 7]{cqqgvro} the there is an edge in $C_Q(\mathbb{O}_0)$ connecting the class
$[\Da_0]$ to itself, i.e., there is an edge in $\mathfrak{T}(K_Q)$ connecting two orders
isomorphic to $\Da_0$. In particular, if the corresponding Eichler order were split, there should exist 
two divisors $B_0$ and $B'_0$ of degree $0$
satisfying $|B_0-B'_0|=Q$, and the result follows as before. 
\end{proof}

\begin{remark}
We let $\Gamma_0=K^*\Da(U)^*/K^*\subseteq\Gamma$, and let 
$S_Q(\Da)=\Gamma_0\backslash\mathfrak{T}(K_Q)$ be the S-graph of $\Da$ as in \cite{cqqgvro}.
 Since $\Gamma_0$ is a normal subgroup of
$\Gamma$, the classifying graph is a quotient of the S-graph. This can be used as a tool to compute classifying 
graphs since the valency, in the S-graph, of the vertex corresponding to an order $\Da'$ is the number of orbits
of the natural action of the group $\Da'(X)^*$ as moebius transformations acting over the set of rational points
of the projective line over the corresponding residue field,  which we identify with the set of $Q$-neighbors 
of $\Da$. This has a particularly simple description for a split Eichler order $\Da=\Ea[B,B']$:
\begin{quote}
Assume that $B+B'$ is effective and non-zero. Then either $B$ or $B'$ has positive degree, say $B$ to fix ideas.
Then $\mathfrak{L}^{-B}(X)=\{0\}$. A  simple computation shows that
$\Ea[B,B'](X)=\bmattrix {\finitum}{\mathfrak{L}^{-B'}(X)}{0}{\finitum}$, and any element whose only
eigenvalue is $1$ acts by conjugation as an aditive map of the form $t\mapsto t+a$ on the projective line
$\mathbb{P}^1\big(\finitum(Q)\big)$.\end{quote}
 We conclude that any vertex in the
S-graph $S_Q(\Da)$ corresponding to such an order has valency $2$ as soon as
$$\dim_{\finitum}\Big(\mathfrak{L}^{-B'}(X)/\mathfrak{L}^{-B'-Q}(X)\Big)=[\finitum(Q):\finitum]=\deg(Q),$$
while its valency is $2+|\finitum(Q)^*/\finitum^*|$ if the preceding dimension is $0$. In particular, if $Q$ is
a point of degree $1$, the valency of a split order can be either $2$ or $3$. It is a consequence of Riemann-Roch' 
Theorem that the valency is always $2$ for large values of $\deg(-B')$.
\end{remark}

\section{Proof of the main results}

We begin this section by proving a few key lemmas.

\begin{lemma}\label{l41}
Let $P_1,P_2,P_3\in |\mathbb{P}^1|$ be three points of degree 1. Then every order in  $\mathbb{O}_{P_1+P_2}$
is split, but there exists a unique conjugacy class of non-split orders in $\mathbb{O}_{P_1+P_2+P_3}$.
\end{lemma}

\begin{proof}
Recall as before that the maximal orders containing a fixed non-trivial idempotent,
i.e., a conjugate of  $\sbmattrix 1000$, are the vertices of a maximal path \cite[Cor. 4.3]{Eichler2}.
 On the other hand, the 
classifying graph (or the S-graph), for maximal orders at a point $P_1$ of degree $1$ is as shown in Figure 1A
(c.f. \cite[\S II.2.3]{trees}, or \cite[Fig. 1]{cqqgvro}).  This is covered twice by the maximal path in Figure 1B.
Every edge of this graph corresponds to a conjugacy class of orders in $\mathbb{O}_{P_1}$ and
conversely, whence every order in this genus is split. In fact, all classes in this genus are represented in the
set $$\{\Ea[P_1,0],\Ea[2P_1,-P_1],\Ea[3P_1,-2P_1],\dots\}.$$
The corresponding quotient graph is easily drawn from the preceding discussion as we can see that 
all vertices have valency $2$ in the S-graph at a place $P_2\neq P_1$ of degree 1, 
thus we obtain the graph in Figure 1C.
Note that $\Ea[(n+1)P_1,-nP_1]\cong \Ea[nP_2+P_1,-nP_2]$.
 The edges in this graph correspond to the orders in  
$\mathbb{O}_{P_1+P_2}$. Again, each of these edges in the quotient graph has a representative in
the maximal path corresponding to the global idempotent  $\sbmattrix 1000$. We conclude that each order in this
genus is split. Representatives for all these orders are in the set
$$\{\Ea[P_1,P_2],\Ea[P_1+P_2,0],\Ea[P_1+2P_2,-P_2],\Ea[P_1+3P_2,-2P_2],\dots\}.$$
In particular, conjugacy classes of such orders can be characterize by the conjugacy classes of 
the four maximal orders containing each order in this genus. For example, the maximal orders containing
the order $\Ea[P_1,P_2]$ have the form $\Da_B$ where $B\leq P_1$ and $-B\leq P_2$,
so $B\in\{ 0,P_1,-P_2,P_1-P_2\}$ and they belong to the classes 
$[\Da_0]$, $[\Da_{P_1}]$, $[\Da_{P_1}]$, and $[\Da_0]$  respectively.

We can iterate this procedure on the classifying graph for $\mathbb{O}_{P_1+P_2}$ (Figure 1D) at a third place
$P_3$. Note that $\Ea[P_1+nP_3,P_2-nP_3]\cong \Ea[nP_2+P_1,(1-n)P_2]$.
\begin{figure}[h]
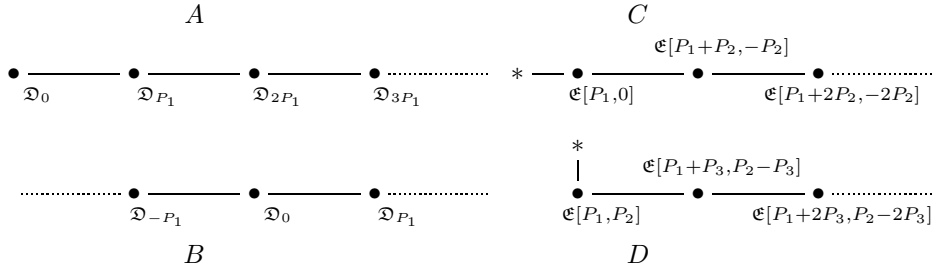

\[  \xygraph{
!{<0cm,0cm>;<.8cm,0cm>:<0cm,.8cm>::}
!{(0,4)}*+{\bullet}="a" !{(0.4,3.6)}*+{{}^{\Da_0}}="a1"
!{(2,4)}*+{\bullet}="b" !{(2.4,3.6)}*+{{}^{\Da_{P_1}}}="b1"
!{(4,4)}*+{\bullet}="c" !{(4.4,3.6)}*+{{}^{\Da_{2P_1}}}="c1"
!{(6,4)}*+{\bullet}="x" !{(6.4,3.6)}*+{{}^{\Da_{3P_1}}}="x1"
!{(8,4)}*+{}="d"
!{(0,2)}*+{}="m"
!{(2,2)}*+{\bullet}="e" !{(2.4,1.6)}*+{{}^{\Da_{-P_1}}}="e1"
!{(4,2)}*+{\bullet}="f" !{(4.4,1.6)}*+{{}^{\Da_0}}="f1"
!{(6,2)}*+{\bullet}="g" !{(6.4,1.6)}*+{{}^{\Da_{P_1}}}="g1"
!{(8,2)}*+{}="h"
!{(3,5)}*+{A}="z" !{(3,1)}*+{B}="z1"
 "a"-"b" "b"-"c" "c"-"x" "x"-@{.}"d"
 "e"-"f" "f"-"g"
"g"-@{.}"h" "m"-@{.}"e"} 
\xygraph{
!{<0cm,0cm>;<.8cm,0cm>:<0cm,.8cm>::}
!{(1,4)}*+{*}="a" 
!{(2,4)}*+{\bullet}="b" !{(2.4,3.6)}*+{{}^{\Ea[P_1,0]}}="b1"
!{(4,4)}*+{\bullet}="c" !{(4.4,4.4)}*+{{}^{\Ea[P_1+P_2,-P_2]}}="c1"
!{(6,4)}*+{\bullet}="x" !{(6.4,3.6)}*+{{}^{\Ea[P_1+2P_2,-2P_2]}}="x1"
!{(8,4)}*+{}="d"
!{(2,2.8)}*+{*}="m"
!{(2,2)}*+{\bullet}="e" !{(2.4,1.6)}*+{{}^{\Ea[P_1,P_2]}}="e1"
!{(4,2)}*+{\bullet}="f" !{(4.4,2.4)}*+{{}^{\Ea[P_1+P_3,P_2-P_3]}}="f1"
!{(6,2)}*+{\bullet}="g" !{(6.4,1.6)}*+{{}^{\Ea[P_1+2P_3,P_2-2P_3]}}="g1"
!{(8,2)}*+{}="h"
!{(3,5)}*+{C}="z" !{(3,1)}*+{D}="z1"
 "a"-"b" "b"-"c" "c"-"x" "x"-@{.}"d"
 "e"-"f" "f"-"g"
"g"-@{.}"h" "m"-"e"} 
\]\caption{Four graphs used in the proof of Lemma 3.4.}\end{figure}
If we try to use this graph to prove that all edges correspond to split orders we find a problem.
The vertex corresponding to $\Ea[P_1,P_2]$ has valency $3$ in the S-graph, with only two
of its edges having representatives in the maximal path, joining this vertex to $\Ea[P_1-P_3,P_2+P_3]$
and $\Ea[P_1+P_3,P_2-P_3]$, both isomorphic to $\Ea[P_1+P_2,0]$ . They are identified in
the classifying graph by the formula
$$\bmattrix 0f10\Ea[B,D]{\bmattrix 0f10}^{-1}=\Ea[D-\mathrm{div}(f),B+\mathrm{div}(f)],$$
by setting $\mathrm{div}(f)=P_2-P_1$.
Any other edge whose starting point is $\Ea[P_1,P_2]$ is in the class corresponding
to the third vertex in the S-graph. Since every vertex in the S-graph corresponding to
a different conjugacy class has valency 2 with non-isomorphic neighbors, this third vertex can only
join two orders isomorphic to $\Ea[P_1,P_2]$. We conclude that the classifying graph looks like
the one in Figure 1D. The vertical half-edge joining $\Ea[P_1,P_2]$ with a virtual endpoint has no representative
on the main maximal path, but 
it might have a representative in the maximal path corresponding to a different global idempotent.
 We must prove that this is not the case.
Assume that the Eichler order $\Ea$ corresponding to this edge has an idempotent global section $\rho$.
We observe that both $P_3$-faces of the corresponding grid correspond to  
conjugates of the order $\Ea[P_1,P_2]$, whence the maximal orders corresponding to each of the
eight vertices belongs to the class shown in Figure 2.
 Assume a basis is chosen in a way that $\rho=\sbmattrix 1000$. 
Conjugating by suitable diagonal matrices, we can assume that one of the vertices 
in the class $[\Da_0]$ is actually $\Da_0$. Then, no choice of
the signs in the neighboring vertices, which must be $\Da_{P_i}$ or $\Da_{-P_i}$ in each case, give us
the configuration of classes shown in Figure 2. This is a contradiction.   
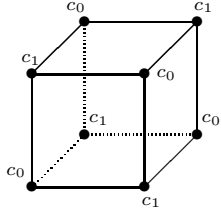
\begin{figure}
\unitlength 1mm % = 2.845pt
\linethickness{0.4pt}
\ifx\plotpoint\undefined\newsavebox{\plotpoint}\fi % GNUPLOT compatibility
\[
\begin{picture}(28,25)(0,0)
\put(0,0){\line(0,1){15}}\put(0,0){\line(1,0){15}}
\put(15,0){\line(0,1){15}}\put(0,15){\line(1,0){15}}
\put(0,15){\line(1,1)7}\put(15,15){\line(1,1)7}
\put(7,22){\line(1,0){15}}\put(22,7){\line(0,1){15}}
\put(0,15){\line(1,1)7}\put(15,15){\line(1,1)7}\put(15,0){\line(1,1)7}
%\multiput(15,15)(-0.1,0.0875){80}{\line(0,-1){.05}}
%\multiput(15,0)(-0.5,0.4375){16}{\line(0,-1){.05}}
\multiput(0,0)(0.4375,0.4375){16}{\line(0,-1){.05}}
\multiput(7,7)(0,0.5){30}{\line(0,1){.05}}
\multiput(7,7)(0.5,0){30}{\line(1,0){.05}}
\put(0,0){\makebox(0,0)[cc]{$\bullet$}}\put(7,7){\makebox(0,0)[cc]{$\bullet$}}
\put(22,7){\makebox(0,0)[cc]{$\bullet$}}\put(15,0){\makebox(0,0)[cc]{$\bullet$}}
\put(0,15){\makebox(0,0)[cc]{$\bullet$}}\put(7,22){\makebox(0,0)[cc]{$\bullet$}}
\put(22,22){\makebox(0,0)[cc]{$\bullet$}}\put(15,15){\makebox(0,0)[cc]{$\bullet$}}
\put(-2,2){\makebox(0,0)[cc]{${}_{c_0}$}}\put(9,9){\makebox(0,0)[cc]{${}_{c_1}$}}
\put(24,9){\makebox(0,0)[cc]{${}_{c_0}$}}\put(16,-2){\makebox(0,0)[cc]{${}_{c_1}$}}
\put(0,17){\makebox(0,0)[cc]{${}_{c_1}$}}\put(6,24){\makebox(0,0)[cc]{${}_{c_0}$}}
\put(23,24){\makebox(0,0)[cc]{${}_{c_1}$}}\put(18,15){\makebox(0,0)[cc]{${}_{c_0}$}}
\end{picture}
\] \caption{Conjugacy classes of the maximal orders containing the only non-split Eichler order
in the genus $\mathbb{O}_{B}$ when $B$ is the sum of three different points of degree 1.
For simplicity we use $c_n=[\Da_{nP_1}]$.}
\end{figure}
\end{proof}

\begin{lemma}\label{l42}
Let $X$ be an arbitrary smooth projective curve, and
let $P\in |X|$ be an arbitrary point. Then there is an infinite
set of conjugacy classes of non-split orders in  $\mathbb{O}_{2P}$.
\end{lemma}

\begin{proof}
Fix an order $\Ea$ of level $2P$ and a maximal order $\Da$ containing $\Ea$. 
 Any cusp in the classifying graph $C_P(\Da)$ looks like in Figure 3A, where each
order in the class $[\Da_{B+nP}]$ has one neighbor in the class $[\Da_{B+(n+1)P}]$ and
all the others in the class $[\Da_{B+(n-1)P}]$. Since the orders $\Ea'\in\mathbb{O}_{2P}$
satisfying $\Ea'(U)=\Da(U)$, where $U=X-\{P\}$, correspond to paths of length $2$
 in the bruhat tits tree, for every
value of $n>1$ there exists an Eichler order contained in one order in the class $[\Da_{B+nP}]$ 
and two orders in the class $[\Da_{B+(n-1)P}]$. We claim that such orders are non-split for $n>-\mathrm{deg}(B)$.
As they are evidently in different conjugacy classes, the result would follow.  Now let $\Ea$
be an Eichler order of level $2P$ whose grid has vertices in the conjugacy classes shown
 in Figure 3B. If $\Ea$ were split, by an appropiate 
choice of coordinates we can assume $\Ea=\Ea[D,D']$ where $D+D'=2P$ or $-D'=D-2P$, whence
the three maximal orders containing $\Ea$ must be $\Da_D$, $\Da_{D-P}$ and $\Da_{D-2P}$,
with $D-P$ linearly equivalent to $B+nP$, and hence of positive degree. 
\begin{figure}[h]
\[  
\unitlength 1mm % = 2.845pt
\linethickness{0.4pt}
\ifx\plotpoint\undefined\newsavebox{\plotpoint}\fi % GNUPLOT compatibility
\begin{picture}(40,15)(0,-6)
\put(0,0){\line(0,1){5}}\put(0,5){\line(1,0){5}}\put(0,0){\line(1,0){5}}\put(5,0){\line(0,1){5}}
\put(5,2.5){\makebox(0,0)[cc]{$\bullet$}}\put(5.2,2.5){\line(1,0){8}}\put(13.7,2.5){\makebox(0,0)[cc]{$\bullet$}}
\put(14.2,2.5){\line(1,0){8}}\put(22.7,2.5){\makebox(0,0)[cc]{$\bullet$}}
\put(23.2,2.5){\line(1,0){8}}\put(31.7,2.5){\makebox(0,0)[cc]{$\bullet$}}
\put(36,2.5){\makebox(0,0)[cc]{$\cdots\cdots$}}
\put(5,-4){\makebox(0,0)[cc]{$A$}}\put(2.5,2.5){\makebox(0,0)[cc]{${}_X$}}
\put(13.7,5){\makebox(0,0)[cc]{${}_{\Da_B}$}}
\put(22.7,0){\makebox(0,0)[cc]{${}_{\Da_{B+P}}$}}
\put(31.7,5){\makebox(0,0)[cc]{${}_{\Da_{B+2P}}$}}
\end{picture}
\qquad\qquad
\xygraph{
!{<0cm,0cm>;<.8cm,0cm>:<0cm,.8cm>::}
!{(2,2)}*+{\bullet}="e" !{(2.4,1.6)}*+{{}^{[\Da_{B+(n-1)P}]}}="e1"
!{(4,2)}*+{\bullet}="f" !{(4.4,2.4)}*+{{}^{[\Da_{B+nP}]}}="f1"
!{(6,2)}*+{\bullet}="g" !{(6.4,1.6)}*+{{}^{[\Da_{B+(n-1)P}]}}="g1"
 !{(3,1)}*+{B}="z1"
 "e"-"f" "f"-"g"} 
\]\caption{Two graphs used in the proof of Lemma \ref{l42}. The square marked "X" denotes a possibly infinite subgraph.}\end{figure}
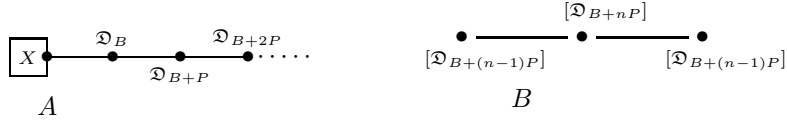
We conclude that the absolute value of the degrees of the divisors $D$ and $D-2P$ 
are different, so the corresponding orders cannot be conjugate. The result follows.
\end{proof}

\begin{remark}
At the end of the preceding proof, we can also prove that $\Ea$ is not split by observing that 
$\Da_{B+(n+1)P}$, as a neighbor of $\Da_{B+nP}=\Da_\Lambda$, correspond to a common eigenspace, in 
$\Lambda_P/\pi_P\Lambda_P$, of every idempotent global section of $\Da_{B+nP}$, whence no idempotent 
in the ring of global sections $\Da_{B+nP}(X)$ has two eigenspaces in $\Lambda_P/\pi_P\Lambda_P$
corresponding to $P$-neighbors isomorphic to $\Da_{B+(n-1)P}$.
\end{remark}

\begin{example}\label{e44}
Let $P$ and $Q$ be points of degree 1 in the proyective line $\mathbb{P}^1$. Consider an order
$\Ea\in\mathbb{O}_{2P}$ and the classifying graph $C_Q(\Ea)$. The vertices of this graph, or equivalently,
the conjugacy classes in $\mathbb{O}_{2P}$ are in correspondence with orbits of paths of length $2$ in the
corresponding Bruhat-Tits tree. By analyzing the orbits
of the action of the group $\Da(X)^*$ of invertible global 
sections of the central vertex $\Da$ of each path, on pairs of elements in the 
corresponding projective line over the residue field, 
we can prove that these vertices correspond precisely 
to simplicial maps $\gamma:\mathfrak{I}_{0,2}
\rightarrow C_P(\Da)$ that can be lifted to paths in $\mathfrak{T}(K_P)$. 
The latter condition rules out the maps satisfying $$\gamma_V(n_0)=\gamma_V(n_2)=[\Da_{(n+1)P}],\qquad
\gamma_V(n_1)=[\Da_{nP}],$$ as such a map has no injective lifting, because $\Da_{nP}$ has a unique
neighbor in the class $[\Da_{(n+1)P}]$. All other maps in Figure 1A can be lifted to paths in $\mathfrak{T}$,
and hence correspond to Eichler orders.
 They are the split orders $\Ea[P,P],\Ea[2P,0],\Ea[3P,-P],\dots$ and the orders 
$\Fa_r$ which are intersections of two orders in the class $[\Da_{rP}]$ and contained in an order in the class
$[\Da_{(r+1)P}]$ as in Figure 3B. 

The edges of the graph  $C_Q(\Ea)$ are in correspondence with the admisible grids like the one in Figure 4A.
By switching the role played by the places $P$ and $Q$, these grids are also in correspondence with simplicial maps
$\gamma:\mathfrak{I}_{0,2}\rightarrow C_P(\Da)$, up to reversal, in the graph in Figure 1C, where again we
must consider only the maps lifting to paths. A few of these grids are shown in Figure 4C-E. 
\begin{figure}[h]
\[  
\unitlength 1mm % = 2.845pt
\linethickness{0.4pt}
\ifx\plotpoint\undefined\newsavebox{\plotpoint}\fi % GNUPLOT compatibility
\begin{picture}(20,15)(0,-6)
\put(0,0){\line(0,1){8}}\put(0,8){\line(1,0){16}}\put(0,0){\line(1,0){16}}\put(8,0){\line(0,1){8}}\put(16,0){\line(0,1){8}}
\put(5,-4){\makebox(0,0)[cc]{$A$}}
\end{picture}
\qquad\qquad
 \xygraph{
!{<0cm,0cm>;<.8cm,0cm>:<0cm,.8cm>::}
!{(0,2)}*+{\bullet}="a" !{(0.4,1.6)}*+{\Fa_1}="a1"
!{(2,2)}*+{\bullet}="b" !{(2.4,1.6)}*+{\Fa_0}="b1"
!{(4,2)}*+{\bullet}="c" !{(4.4,1.6)}*+{{}^{\Ea[P,P]}}="c1"
!{(6,2)}*+{\bullet}="x" !{(6.4,1.6)}*+{{}^{\Ea[2P,0]}}="x1"
!{(8,2)}*+{}="d"
!{(-2,2)}*+{}="m"
!{(3,3)}*+{B}="z" 
 "a"-"b" "b"-"c" "c"-"x" "x"-@{.}"d"
 "m"-@{.}"a"}  
\]
\[
\unitlength 1mm % = 2.845pt
\linethickness{0.4pt}
\ifx\plotpoint\undefined\newsavebox{\plotpoint}\fi % GNUPLOT compatibility
\begin{picture}(20,17)(0,-6)
\put(0,0){\line(0,1){8}}\put(0,8){\line(1,0){16}}\put(0,0){\line(1,0){16}}\put(8,0){\line(0,1){8}}\put(16,0){\line(0,1){8}}
\put(0,0){\makebox(0,0)[cc]{$\bullet$}}\put(0,8){\makebox(0,0)[cc]{$\bullet$}}
\put(8,0){\makebox(0,0)[cc]{$\bullet$}}\put(8,8){\makebox(0,0)[cc]{$\bullet$}}
\put(16,0){\makebox(0,0)[cc]{$\bullet$}}\put(16,8){\makebox(0,0)[cc]{$\bullet$}}
\put(2,1){\makebox(0,0)[cc]{${}^{c_0}$}}\put(2,9){\makebox(0,0)[cc]{${}^{c_1}$}}
\put(10,1){\makebox(0,0)[cc]{${}^{c_1}$}}\put(10,9){\makebox(0,0)[cc]{${}^{c_2}$}}
\put(18,1){\makebox(0,0)[cc]{${}^{c_0}$}}\put(18,9){\makebox(0,0)[cc]{${}^{c_1}$}}
\put(5,-4){\makebox(0,0)[cc]{$C$}}
\end{picture}\qquad
\unitlength 1mm % = 2.845pt
\linethickness{0.4pt}
\ifx\plotpoint\undefined\newsavebox{\plotpoint}\fi % GNUPLOT compatibility
\begin{picture}(20,17)(0,-6)
\put(0,0){\line(0,1){8}}\put(0,8){\line(1,0){16}}\put(0,0){\line(1,0){16}}\put(8,0){\line(0,1){8}}\put(16,0){\line(0,1){8}}
\put(0,0){\makebox(0,0)[cc]{$\bullet$}}\put(0,8){\makebox(0,0)[cc]{$\bullet$}}
\put(8,0){\makebox(0,0)[cc]{$\bullet$}}\put(8,8){\makebox(0,0)[cc]{$\bullet$}}
\put(16,0){\makebox(0,0)[cc]{$\bullet$}}\put(16,8){\makebox(0,0)[cc]{$\bullet$}}
\put(2,1){\makebox(0,0)[cc]{${}^{c_1}$}}\put(2,9){\makebox(0,0)[cc]{${}^{c_0}$}}
\put(10,1){\makebox(0,0)[cc]{${}^{c_0}$}}\put(10,9){\makebox(0,0)[cc]{${}^{c_1}$}}
\put(18,1){\makebox(0,0)[cc]{${}^{c_1}$}}\put(18,9){\makebox(0,0)[cc]{${}^{c_0}$}}
\put(5,-4){\makebox(0,0)[cc]{$D$}}
\end{picture}\qquad
\unitlength 1mm % = 2.845pt
\linethickness{0.4pt}
\ifx\plotpoint\undefined\newsavebox{\plotpoint}\fi % GNUPLOT compatibility
\begin{picture}(20,17)(0,-6)
\put(0,0){\line(0,1){8}}\put(0,8){\line(1,0){16}}\put(0,0){\line(1,0){16}}\put(8,0){\line(0,1){8}}\put(16,0){\line(0,1){8}}
\put(0,0){\makebox(0,0)[cc]{$\bullet$}}\put(0,8){\makebox(0,0)[cc]{$\bullet$}}
\put(8,0){\makebox(0,0)[cc]{$\bullet$}}\put(8,8){\makebox(0,0)[cc]{$\bullet$}}
\put(16,0){\makebox(0,0)[cc]{$\bullet$}}\put(16,8){\makebox(0,0)[cc]{$\bullet$}}
\put(2,1){\makebox(0,0)[cc]{${}^{c_2}$}}\put(2,9){\makebox(0,0)[cc]{${}^{c_1}$}}
\put(10,1){\makebox(0,0)[cc]{${}^{c_1}$}}\put(10,9){\makebox(0,0)[cc]{${}^{c_0}$}}
\put(18,1){\makebox(0,0)[cc]{${}^{c_0}$}}\put(18,9){\makebox(0,0)[cc]{${}^{c_1}$}}
\put(5,-4){\makebox(0,0)[cc]{$E$}}
\end{picture}
\]
\caption{The domino-shaped grid (A) used to compute the graph
in Ex. \ref{e44} (B). In (C)-(E) we have the grids corresponding to the three central edges in (B).
Again, we use $c_n=[\Da_{nP}]$.}\end{figure}
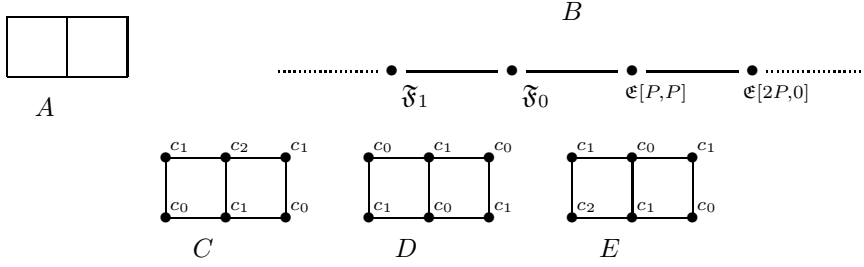
We conclude that the graph $C_Q(\Ea)$ looks as in Figure 4B. 
\end{example}

\subparagraph{Proof of Theorem  \ref{t2}}
Let $D$ be an effective divisor on $\mathbb{P}^1$. First we assume $D$ is the sum of at most two
places of degree $1$. Then $D\leq P_1+P_2$ for some pair of places $P_1$ and $P_2$ of degree 1.
By looking at the product of the local Bruhat-Tits trees at $P_1$ and $P_2$ we observe that 
any order in $\mathbb{O}_D$ corresponds to a vertex, edge, or grid contained in the $1$-times-$1$
grid corresponding to an order $\Ea'\in\mathbb{O}_{P_1+P_2}$. This is a split order, as shown in Lemma 
\ref{l41}, whence its ring $\Ea'(X)$ of global sections contains a non-trivial idempotent.
Since $\Ea(X)\supseteq\Ea'(X)$, the same holds for $\Ea$ and the result follows. In any other case,
$D\geq B$ for a divisor $B$ in one of the following cases: 
\begin{enumerate}
\item $B=2P$, where $\mathop{\mathrm{deg}}(P)=1$,
\item $B=P_1+P_2+P_3$, where $\mathop{\mathrm{deg}}(P_1)=
\mathop{\mathrm{deg}}(P_2)=\mathop{\mathrm{deg}}(P_3)=1$, or 
\item $B=P$, where $P$ is a place satisfying $\mathop{\mathrm{deg}}(P)>1$.
\end{enumerate}
Then the result follows from Lemma \ref{l42}, Lemma \ref{l41}, or Lemma \ref{l34}, by a similar reasoning.
\qed 

\begin{lemma}
Let $P_1,\dots,P_n\in |X|$ be different. Then there is an finite
set of conjugacy classes of non-split orders in  $\mathbb{O}_{P_1+\cdots+P_n}$.
\end{lemma}

\begin{proof}
We prove this by induction on $n$. This was proved by Serre \cite[Th. II.9]{trees} for the genus
of Maximal orders ($n=0$), so we assume it holds for $n=t$ and prove it for $n=t+1$.
 Conjugacy classes in $\mathbb{O}_{P_1+\cdots+P_t}$ are in correspondence with the vertices in 
the classifying graph $C_{P_{t+1}}(\mathbb{O}_{P_1+\cdots+P_t})$, so all but a finite number of them
correspond to the conjugacy class of an order $\Ea[B,B']$, where $B+B'=P_1+\cdots+P_t$. By switching
$B$ and $B'$ if needed, we can assume 
$\mathop{\mathrm{deg}}(B)\leq\mathop{\mathrm{deg}}(B')$.  Furthermore, a second
order $\Ea[B'',B''']\in\mathbb{O}_{P_1+\cdots+P_t}$ with $B''$ linearly equivalent to $B$ is in the same 
conjugacy class, so by leaving out a finite number of conjugacy classes, we can always assume
$\mathop{\mathrm{deg}}(B)<M$ for any prescribed constant $M$. In particular, we can assume also
that $B'$ has positive degree, and therefore $\mathfrak{L}^{-B'}(X)=\{0\}$, while
$$\mathrm{dim}_{\finitum}\Big(\mathfrak{L}^{-B}(X)/\mathfrak{L}^{-B-P_{t+1}}(X)\Big)=
[\finitum(P_{t+1}):\finitum]$$
 by Riemann-Roch' Theorem.  We conclude that $\Ea[B,B'](X)^*$ 
has two orbits on the set of neighbors of $\Ea[B,B']$ by the remark at the end of last section. In particular
the corresponding vertex on $C_{P_{t+1}}(\mathbb{O}_{P_1+\cdots+P_t})$ has valency one or two,
and therefore every admissible grid having the grid corresponding to $\Ea[B,B']$ as a $P_{t+1}$-cap
corresponds to either $\Ea[B+P_{t+1},B']$ or $\Ea[B,B'+P_{t+1}]$, in either case it is split. As every
admissible $(P_1+\cdots+P_t)$-grid is the $P_{t+1}$-cap of a finite number of admissible 
$(P_1+\cdots+P_t+P_{t+1})$-grids, the result follows.
\end{proof}

\subparagraph{Proof of Theorem  \ref{t1}}
 If $D$ a multiplicity-free effective divisor, then the result follows from the preceding lemma. Assume now
that $D$ is not multiplicity-free. Then there is a place $P\in|X|$ satisfying $2P\leq D$. It follows that every
order in $\mathbb{O}_{2P}$ contains an order in $\mathbb{O}_D$. Now the result follows from two 
observations:
\begin{enumerate}
\item Every order containing a split order is split.
\item Every order in $\mathbb{O}_D$ is contained in a finite number of orders in $\mathbb{O}_{2P}$.
\end{enumerate}
The first statement follows as in the proof of Theorem \ref{t2}.
The second statement is an immediate consequence of the combinatorial structure of products of 
Bruhat-Tits trees. We conclude that
there is an infinite number of non-conjugate orders in  $\mathbb{O}_D$ contained in non-split orders
in $\mathbb{O}_{2P}$, whence the result follows. \qed
 
\section{Acnowledgements}
This work was suported by Fondecyt, Grant No 1160603.

\end{document}